\documentclass[english,11pt, a4paper]{article}
\usepackage[T1]{fontenc}
\usepackage[latin9]{inputenc}
\usepackage{geometry}
\usepackage{color}
\usepackage{babel}
\usepackage{units}
\usepackage{mathtools}
\usepackage{enumitem}
\usepackage{algorithm2e}
\usepackage{amsmath}
\usepackage{amsthm}
\usepackage{amssymb}
\usepackage{graphicx}
\usepackage[authoryear,round]{natbib}
\usepackage{microtype}
\usepackage[unicode=true,
 bookmarks=true,bookmarksnumbered=false,bookmarksopen=false,
 breaklinks=false,pdfborder={0 0 0},pdfborderstyle={},backref=page,colorlinks=true]
 {hyperref}
\hypersetup{pdftitle={Stochastic Optimization with Kernels},
 pdfauthor={Alois Pichler},
 citecolor= ForestGreen, urlcolor= blue, linkcolor= ForestGreen}

\makeatletter
\numberwithin{equation}{section}
\theoremstyle{plain}
\newtheorem{thm}{\protect\theoremname}[section]
\theoremstyle{definition}
\newtheorem{defn}[thm]{\protect\definitionname}
\theoremstyle{remark}
\newtheorem{rem}[thm]{\protect\remarkname}
\theoremstyle{definition}
\newtheorem{example}[thm]{\protect\examplename}
\theoremstyle{plain}
\newtheorem{prop}[thm]{\protect\propositionname}
\theoremstyle{plain}
\newtheorem{cor}[thm]{\protect\corollaryname}

\mathtoolsset{showonlyrefs}

\usepackage{dsfont}   
\usepackage{newtxtext, newtxmath}	
\usepackage{microtype}
\usepackage{siunitx}
\usepackage{tikz}
\usetikzlibrary{matrix,arrows.meta}

\usepackage{doi}
\usepackage{orcidlink}
\definecolor{ForestGreen}{HTML}{005F50} 

\setlist[enumerate,1]{label=(\roman*), ref=(\roman*)}

\DeclareMathOperator{\E}{{\mathds E}}
\DeclareMathOperator{\sign}{{sign}}
\DeclareMathOperator{\VaR}{\mathsf{V@R}}	
\DeclareMathOperator{\AVaR}{\mathsf{AV@R}}	

\DeclareMathOperator{\supp}{{supp}}
\DeclareMathOperator*{\essinf}{ess\,inf}
\DeclareMathOperator*{\esssup}{ess\,sup}
\DeclareMathOperator*{\argmin}{arg\,min}
\newcommand{\one}{{\mathds 1}} 		

\makeatother

\providecommand{\corollaryname}{Corollary}
\providecommand{\definitionname}{Definition}
\providecommand{\examplename}{Example}
\providecommand{\propositionname}{Proposition}
\providecommand{\remarkname}{Remark}
\providecommand{\theoremname}{Theorem}

\begin{document}
\title{\textbf{Expectiles In Risk Averse Stochastic Programming and Dynamic~Optimization}
}
\author{Rajmadan Lakshmanan\thanks{University of Technology, Chemnitz, Faculty of mathematics. 90126
Chemnitz, Germany}  \and  Alois Pichler\footnotemark[1]\, \thanks{\orcidlink{0000-0001-8876-2429}~\protect\href{https://orcid.org/0000-0001-8876-2429}{orcid.org/0000-0001-8876-2429}.
Contact: \protect\href{mailto:alois.pichler@math.tu-chemnitz.de}{alois.pichler@math.tu-chemnitz.de}\protect \\
DFG, German Research Foundation \textendash{} Project-ID 416228727
\textendash{} SFB~1410}}
\maketitle
\begin{abstract}
This paper features expectiles in dynamic and stochastic optimization.
Expectiles are a family of risk functionals characterized as minimizers of optimization problems.
For this reason, they enjoy various unique stability properties, which can
be exploited in risk averse management, in stochastic optimization
and in optimal control.

The paper provides tight relates of expectiles to other risk functionals
and addresses their properties in regression. Further, we extend expectiles
to a dynamic framework. As such, they allow incorporating a risk averse
aspect in continuous-time dynamic optimization and a risk averse variant of the Hamilton\textendash Jacobi\textendash Bellman equations.

\medskip{}


\noindent \textbf{Classification:} 90C08, 90C15, 60G07
\end{abstract}

\section{Introduction\label{sec:Introduction}}

Classical dynamic programming problems involve the expectation in
the objective. The expectation is a risk neutral assessment of random
outcomes. In many situations, specifically in economic environments,
a risk averse assessment or risk management is much more favorable
and desirable. For this reason there have been attempts to develop
risk averse dynamic programming principles and risk averse Hamilton\textendash Jacobi\textendash Bellman
equations. 

Non-linear expectations ($g$\nobreakdash-expectations, cf.\ \citet{Peng,Peng2002,Peng1992,Peng2004,Peng2010})
have been considered, e.g., to incorporate the aspect of risk to dynamic
equations. A seemingly simpler approach involves risk measures (or
risk functionals) instead of non-linear expectations, as risk measures
are able to assess the risk associated with a random outcome (cf.\ \citet{RuszczynskiHJB,Ruszczynski2020}).
By construction, risk measures are defined on random variables. For
dynamic programming, they need to be extended to stochastic processes.
The increments of stochastic processes are random variables so that
composing risk measures over time and accumulating the corresponding
risk is a promising approach to extend risk functionals from random
variables to stochastic processes. 

\medskip{}
Specifically, this paper addresses expectiles in stochastic and dynamic
optimization. Expectiles constitute a family of risk measure with
unique properties. We demonstrate how they can be employed to incorporate
risk aversion in dynamic programming and to develop risk averse Hamilton\textendash Jacobi\textendash Bellman
equations.

\citet{ContRobustness} point out the importance of estimating risk
measures in a robust way. In this context, \citet{Gneiting} proves
that the Average Value-at-Risk, the most important risk measure in
theory and practice, is not elicitable, that is, it is not possible
to describe the risk measure as minimizer. More generally, \citet{Ziegel2014}
proves that the only elicitable spectral risk measure is the (trivial)
expectation. \citet{Bellini2014} finally provide a proof that only
expectiles constitute elicitable risk measures. 

Expectiles have been introduced earlier in \citet{NeweyPowell} as
\begin{equation}
e_{\alpha}(X)\coloneqq\argmin_{x\in\mathbb{R}}\E\ell_{\alpha}(X-x),\label{eq:Argmin}
\end{equation}
where $X$ is a $\mathbb{R}$\nobreakdash-valued random variable,
$\alpha\in(0,1)$ and the scoring function (loss function) is\footnote{$x_{+}\coloneqq\max(0,x)$}
\begin{equation}
\ell_{\alpha}(x)\coloneqq\alpha\cdot x_{+}^{2}+(1-\alpha)(-x)_{+}^{2}=\begin{cases}
\quad\alpha\cdot x^{2} & \text{if }x\ge0,\\
(1-\alpha)\,x^{2} & \text{if }x\le0.
\end{cases}\label{eq:ell}
\end{equation}
The characterization as a minimizer in the definition~\eqref{eq:Argmin}
applies for $X\in L^{2}$. The first order condition (cf.~\eqref{eq:ExpectileFO}
below) is an equivalent characterization of the expectile, which applies
\textendash{} more generally \textendash{} for $X\in L^{1}\supset L^{2}$. 
\begin{defn}[Expectiles, cf.~\citep{NeweyPowell}]
\label{def:Expectile}For $X\in L^{1}$ and a risk level $\alpha\in[0,1]$,
the expectiles of a random variable~$X$ is the unique solution of
the equation 
\begin{equation}
\alpha\,\E(X-x)_{+}=(1-\alpha)\E(x-X)_{+},\label{eq:ExpectileFO}
\end{equation}
where $x\in\mathbb{R}$.
\end{defn}

\begin{rem}
In an alternative way, replacing the objective in~\eqref{eq:Argmin}
by $\E\bigl(\ell_{\alpha}(X-x)-\ell_{\alpha}(X-x_{0})\bigr)$ for
some fixed $x_{0}\in\mathbb{R}$ extends the definition to $X\in L^{1}$
as well, so that expectiles are well-defined for $X\in L^{1}$, even
as minimizers.
\end{rem}

For $\alpha=\nicefrac{1}{2},$ the expectile is the expectation, $e_{\nicefrac{1}{2}}(X)=\E X$.
It follows from symmetry of the loss function $\ell_{\alpha}$ (i.e.,
$\ell_{\alpha}(x)=\ell_{1-\alpha}(-x)$) that 
\begin{equation}
e_{\alpha}(X)=-e_{1-\alpha}(-X),\label{eq:Symmetry}
\end{equation}
so that the expectile involves both tails, the lower and the upper
tail of the distribution of the random variable~$X$. For $X\in L^{\infty}$,
the expectile approaches the essential supremum for increasing risk
level, $e_{\alpha}(X)\to\esssup X$ as~$\alpha\to1$.\footnote{\label{fn:esssup}The essential supremum of $X$ is the smallest \emph{number}
$c\in\mathbb{R}$ so that $X\le c$ a.s.} More generally, we have the monotone behavior  
\begin{equation}
\E X\le e_{\alpha}(X)\le e_{\alpha^{\prime}}(X)\le\esssup X\label{eq:5}
\end{equation}
for $\nicefrac{1}{2}\le\alpha\le\alpha^{\prime}\le1$.

\paragraph{Outline of the paper.}

In the following Section~\ref{sec:Elicitable} we elaborate that
expectiles constitute a risk measure, and we provide tight relations
to other risk measures. Next, we introduce conditional risk functionals
in Section~\ref{sec:Conditional}. These are important for risk management
in discrete and in continuous time. In continuous time (Section~\ref{sec:Processes}),
we consider the risk-averse generator, which turns out to be a non-linear
differential operator. We finally employ expectiles for dynamic optimization
problems in Section~\ref{sec:Control} and conclude in Section~\ref{sec:Summary}.

\section{\label{sec:Elicitable}Elicitable risk measures}

The expectile $e_{\alpha}(\cdot)$ is a risk measure as introduced
in \citet{Artzner1999}. That is, the mapping $X\mapsto e_{\alpha}(X)$,
provided that $\alpha\ge\nicefrac{1}{2}$, satisfies the following
four axioms formulated for (convex) risk measures $\mathcal{R}\colon\mathcal{Y}\to\mathbb{R}$,
where~$\mathcal{Y}$ is an appropriate linear space of $\mathbb{R}$\nobreakdash-valued
random variables (for example $\mathcal{Y}=L^{1}(P)$) on the probability
space $(\Omega,\mathcal{F},P)$:
\begin{enumerate}[noitemsep]
\item \label{enu:1} $\mathcal{R}(X)\le\mathcal{R}(Y)$ for all $X\le Y$
almost everywhere, 
\item \label{enu:2} $\mathcal{R}(X+Y)\le\mathcal{R}(X)+\mathcal{R}(Y)$
for all $X$, $Y\in\mathcal{Y}$,
\item \label{enu:3} $\mathcal{R}(\lambda\,X)=\lambda\,\mathcal{R}(X)$
for all $\lambda>0$, and
\item \label{enu:4} $\mathcal{R}(c+X)=c+\mathcal{R}(X)$ for all $c\in\mathbb{R}$. 
\end{enumerate}
The expectile is a risk functional satisfying the Axioms~\ref{enu:1}\textendash \ref{enu:4}
above (Appendix~\ref{sec:Appendix} presents a brief proof for the
subadditivity~\ref{enu:2}, while the other assertions are evident).
Further, the expectile $e_{\alpha}(\cdot)$ is the only risk measure
which can be expressed as a minimizer \textendash ~as in~\eqref{eq:Argmin}~\textendash{}
in addition. We will elaborate below that the expectile is not a spectral
risk measure. The natural space (cf.\ \citet{Pichler2013a}) of expectiles
is $\mathcal{Y}=L^{1}$, cf.\ also the discussion in Section~\ref{sec:Introduction}
above. In what follows \textendash ~unless stated differently~\textendash{}
we will always assume that $\mathcal{Y}=L^{1}$.

Explicit expressions for the expectiles are available only in exceptional
cases. For the uniform distribution in the interval $[0,1]$, $U\sim\mathcal{U}[0,1]$,
e.g., the expectile is $e_{\alpha}(U)=\frac{\alpha-\sqrt{\alpha(1-\alpha)}}{2\alpha-1}$. 

To extend expectiles to a risk measure in continuous time employing
the Wiener process (Brownian motion), we shall frequently need the
expectile of the normal distribution, for which at least the following
series expansion is available.
\begin{example}
An explicit expression for the expectile of normally distributed random
variables, $X\sim\mathcal{N}(\mu,\sigma^{2})$, is not available.
It holds that 
\begin{equation}
e_{\alpha}(X)=\mu+\sigma\sqrt{\frac{8}{\pi}}\Bigl(\alpha-\frac{1}{2}\Bigr)+\sigma\frac{8\sqrt{2}}{\sqrt{\pi}^{3}}\Bigl(\alpha-\frac{1}{2}\Bigr)^{3}+\mathcal{O}\Bigl(\alpha-\frac{1}{2}\Bigr)^{5}.\label{eq:13}
\end{equation}
\end{example}

\begin{proof}
The general assertion derives from the standard normal distribution.
Denoting the density of the standard normal distribution by $\varphi(t)=\frac{1}{\sqrt{2\pi}}e^{-t^{2}/2}$
and by $\Phi(x)=\int_{-\infty}^{x}\varphi(t)\,dt$ its antiderivative,
it holds that 
\begin{align*}
\E(X-t)_{+} & =\int_{t}^{\infty}(x-t)\varphi(x)\,dx=\varphi(t)-t\bigl(1-\Phi(t)\bigr)\\
\shortintertext{and}\E(t-X)_{+} & =\int_{-\infty}^{t}(t-x)\varphi(x)\,dx=t\,\Phi(t)+\varphi(t),
\end{align*}
which follows readily by employing the identity $\varphi'(x)=-x\,\varphi(x)$.
Based on~\eqref{eq:ExpectileFO} define now 
\begin{align}
f(\alpha,e) & \coloneqq\alpha\cdot\bigl(\varphi(e)-e\bigl(1-\Phi(e)\bigr)\bigr)-(1-\alpha)\cdot\bigl(e\,\Phi(e)+\varphi(e)\bigr)\label{eq:18}\\
 & =(2\alpha-1)\bigl(\varphi(e)+e\,\Phi(e)\bigr)-\alpha\,e\nonumber 
\end{align}
so that the expectile $e_{\alpha}$ of the normally distributed random
variable $X$ satisfies $f(\alpha,e_{\alpha})=0$ for every $\alpha\in(0,1)$.
We now apply the implicit function theorem.

As $e_{\nicefrac{1}{2}}(X)=\E X=0$ for the normal distribution it
holds that $f(\nicefrac{1}{2},e_{\nicefrac{1}{2}})=0$. Further, the
partial derivatives of $f$ at $(\alpha,e)$ are $f_{\alpha}(\nicefrac{1}{2},0)=\sqrt{\frac{2}{\pi}}$
and $f_{e}(\alpha,e)=-\frac{1}{2}$ so that the first term in assertion~\eqref{eq:13}
follows with the implicit function theorem. The coefficient for the
next term $\bigl(\alpha-\nicefrac{1}{2}\bigr)^{2}$ is zero, because
the function~\eqref{eq:18} is odd with respect to the center $\nicefrac{1}{2}$,
as the normal distribution is symmetric, cf.~\eqref{eq:Symmetry}.
The remaining coefficient is found by differentiating the function~\eqref{eq:18}
further. We omit the rather technical computations here, as our further
results build on the first two terms only.
\end{proof}
\begin{example}
For a log-normal random variable $X$ with $\log X\sim\mathcal{N}(\mu,\sigma^{2})$,
the expectiles are 
\[
e_{\alpha}(X)=e^{\mu+\frac{\sigma^{2}}{2}}+\left(e^{\sigma^{2}}-1\right)e^{2\mu+\sigma^{2}}\,\Bigl(\alpha-\frac{1}{2}\Bigr)\,4\sqrt{e}\bigl(2\Phi(\nicefrac{1}{2})-1\bigr)+\mathcal{O}\Bigl(\alpha-\frac{1}{2}\Bigr)^{2}.
\]
\end{example}

\begin{proof}
As above, the proof again relies on explicitly available expressions
\begin{align*}
\E(X-t)_{+} & =\int_{\log t}^{\infty}\bigl(e^{x}-t\bigr)\varphi(x)\,dx=\sqrt{e}\,\Phi(1-\log t)-t\,\Phi(-\log t)\\
\shortintertext{and}\E(t-X)_{+} & =\int_{-\infty}^{\log t}\bigl(t-e^{x}\bigr)\varphi(x)\,dx=\sqrt{e}\,\Phi(1-\log t)-\sqrt{e}+t\,\Phi(\log t).
\end{align*}
The statement follows again by the implicit function theorem.
\end{proof}

\subsection{Tight comparison with important risk measures}

In what follows, we shall compare expectiles with important risk measures
and give the tightest-possible estimates and the smallest spectral
risk measure enveloping the expectiles.

The Average Value-at-Risk is the smallest convex envelope of the Value-at-Risk
(cf.\ \citet{Follmer2004}). The Average Value-at-Risk can be stated
in the equivalent forms (cf.\ \citet{Pflug2000})
\begin{align}
\AVaR_{\alpha}(X) & \coloneqq\frac{1}{1-\alpha}\int_{\alpha}^{1}F_{X}^{-1}(\alpha)\mathrm{d}\alpha\nonumber \\
 & =\min\left\{ q+\frac{1}{1-\alpha}\E(X-q)_{+}\colon q\in\mathbb{R}\right\} ,\label{eq:AVaR}\\
\shortintertext{\text{where}}\VaR_{\alpha}(X) & \coloneqq F_{X}^{-1}(\alpha)\coloneqq\inf\bigl\{ x\colon P(X\le x)\ge\alpha\bigr\}\label{eq:VaR}
\end{align}
is the Value-at-risk. 

The Average Value-at-risk is the fundamental building block in the
Kusoka representation (cf.\ \citet{Kusuoka}) and the most important
risk functional in actuarial practice. Notice as well that the Average
Value-at-Risk is the \emph{minimum objective} of an optimization problem
(problem~\eqref{eq:AVaR}), while the expectile in~\eqref{eq:Argmin}
is the \emph{minimizer} of an optimization problem.
\begin{rem}[Quantiles]
Similarly to the expectile, the Value-at-Risk defined in~\eqref{eq:VaR}
is a minimizer of an optimization problem, specifically the problem
\begin{align}
\min_{q\in\mathbb{R}}\E\tilde{\ell}_{\alpha} & (X-q)\nonumber \\
\shortintertext{\text{with scoring function}}\tilde{\ell}_{\alpha}(x) & \coloneqq\begin{cases}
-(1-\alpha)\,x & \text{if }x\le0,\\
\quad\alpha\cdot x & \text{if }x\ge0
\end{cases}\ =\Big(\alpha-\frac{1}{2}\Big)x+\frac{1}{2}\left|x\right|,
\end{align}
well-known from quantile regression. Indeed, the first order condition
is $0=\frac{\partial}{\partial q}\E\ell_{\alpha}(X-q)=\alpha\E\one_{\left\{ X>q\right\} }-(1-\alpha)\E\one_{\left\{ X\le q\right\} }=\alpha-P(X\le q)$
and hence the assertion. However, by violating~\ref{enu:2} above,
the Value-at-Risk is \emph{not} a convex risk functional.
\end{rem}

\begin{defn}[Spectral risk measure, cf.\ \citet{Acerbi2002,Acerbi2002a}]
Let $\sigma\colon[0,1)\to\mathbb{R}_{\ge0}$ be a non-negative, non-decreasing
function with $\int_{0}^{1}\sigma(u)\,\mathrm{d}u=1$. Then
\[
\mathcal{R}_{\sigma}(X)=\int_{0}^{1}F_{X}^{-1}(\alpha)\sigma(\alpha)\,\mathrm{d}\alpha,\qquad X\in\mathcal{Y},
\]
is a risk measure. $\mathcal{R}_{\sigma}$ is called a the \emph{spectral
risk measure} and the function~$\sigma$ is called the \emph{spectrum}
of~$\mathcal{R}_{\sigma}$. 
\end{defn}

The expectiles are not a spectral risk measure themselves.
But for every expectile, there is a smallest spectral risk measure. 
\begin{prop}[Enveloping risk measure]
\label{prop:Risk}If $\mathcal{R}_{\sigma}(X)$ is any spectral risk
measure with 
\begin{equation}
e_{\alpha}(X)\le\mathcal{R}_{\sigma}(X)\label{eq:1-1}
\end{equation}
for every random variable~$X\in\mathcal{Y}$, then $e_{\alpha}(X)\le s_{\alpha}(X)\le\mathcal{R}_{\sigma}(X)$
for all $X$, where 
\begin{equation}
s_{\alpha}(X)\coloneqq\int_{0}^{1}F_{X}^{-1}(u)\frac{\alpha(1-\alpha)}{\big(\alpha-u(2\alpha-1)\big)^{2}}\,\mathrm{d}u;\label{eq:8}
\end{equation}
that is, $s_{\alpha}$ is the smallest spectral risk measure larger
than~$e_{\alpha}$.

\begin{figure}
\centering{}\includegraphics[width=0.7\textwidth]{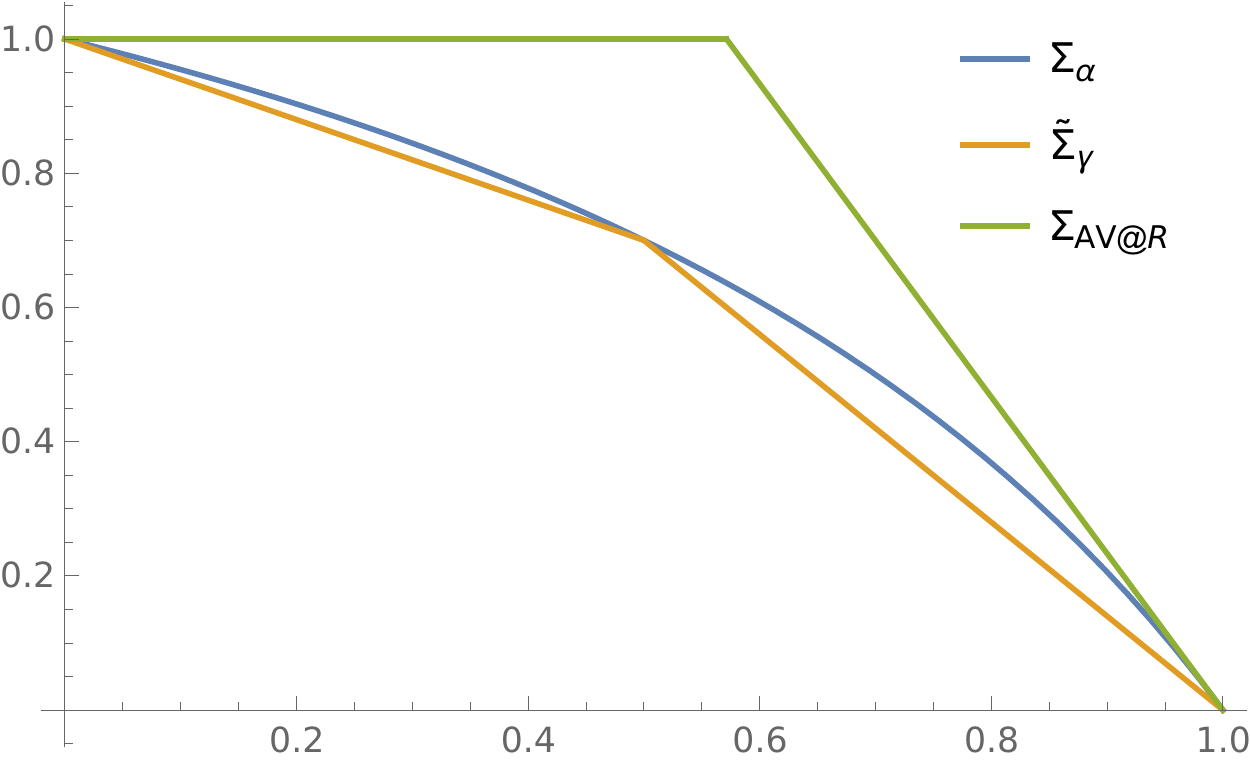}\caption{The function $\Sigma_{\alpha}$ and $\tilde{\Sigma}_{\gamma}$, exemplified
for $\alpha=70\,\%$ and $\gamma=60\,\%$\label{fig:S}}
\end{figure}
\end{prop}

\begin{proof}
Above all, $s_{\alpha}(\cdot)$ is a spectral risk functional, as
$u\mapsto\frac{\alpha(1-\alpha)}{\bigl(\alpha-u(2\alpha-1)\big)^{2}}$
is a non-negative, increasing function and $\int_{0}^{1}\frac{\alpha(1-\alpha)}{\big(\alpha-u(2\alpha-1)\big)^{2}}\,\mathrm{d}u=1$.

\citet[Proposition~9]{Bellini2014} provide the Kusuoka representation
\begin{equation}
e_{\alpha}(X)=\max_{\gamma\in[\nicefrac{1}{\beta},1]}\gamma\E X+(1-\gamma)\AVaR_{\frac{\beta-\frac{1}{\gamma}}{\beta-1}}(X)\label{eq:Kusuoka}
\end{equation}
for expectiles, where $\beta=\frac{\alpha}{1-\alpha}$. Define the
functions $\Sigma_{\gamma}(u)\coloneqq\gamma(1-u)+(1-\gamma)\min\bigg(1,\frac{1-u}{1-\frac{\beta-\frac{1}{\gamma}}{\beta-1}}\bigg)$
and $\Sigma(u)\coloneqq\frac{\alpha(1-u)}{\alpha-u(2\alpha-1)}$.
Both functions coincide at $u=0$, $u=1$ and $u=\frac{\alpha(1+\gamma)-1}{(2\alpha-1)\gamma}$;
indeed $\Sigma_{\gamma}(0)=\Sigma(0)=1$, $\Sigma_{\gamma}(1)=\Sigma(1)=0$
and 
\begin{equation}
\Sigma\left(\frac{\alpha(1+\gamma)-1}{(2\alpha-1)\gamma}\right)=\Sigma_{\gamma}\left(\frac{\alpha(1+\gamma)-1}{(2\alpha-1)\gamma}\right)=\frac{\alpha(1-\gamma)}{2\alpha-1}.\label{eq:17}
\end{equation}
As $\Sigma_{\gamma}$ is piece wise linear and $\Sigma$ concave, it
follows that $\Sigma_{\gamma}(u)\le\Sigma(u)$ for all $u\in[0,1]$.
With integration by parts it follows further that 
\begin{align}
\gamma\E X+(1-\gamma)\AVaR_{\frac{\beta-\frac{1}{\gamma}}{\beta-1}}(X) & =-\int_{0}^{1}F_{X}^{-1}(u)\,\mathrm{d}\Sigma_{\gamma}(u)\nonumber \\
 & =-\left.F_{X}^{-1}(u)\Sigma_{\gamma}(u)\right|_{u=0}^{1}+\int_{0}^{1}\Sigma_{\gamma}(u)\,\mathrm{d}F_{X}^{-1}(u)\nonumber \\
 & \le-\left.F_{X}^{-1}(u)\Sigma(u)\right|_{u=0}^{1}+\int_{0}^{1}\Sigma(u)\,\mathrm{d}F_{X}^{-1}(u)\label{eq:3-2}\\
 & =-\int_{0}^{1}F_{X}^{-1}(u)\,\mathrm{d}\Sigma(u)\nonumber \\
 & =s_{\alpha}(X)\nonumber 
\end{align}
and thus $e_{\alpha}\le s_{\alpha}$. The assertion follows, as for
every $u\in(0,1)$ there is $\gamma\in\bigl(\frac{1-\alpha}{\alpha\thinspace},1\bigr)$
($\gamma=\frac{1-\alpha}{u(1-2\alpha)+\alpha}$) so that $\Sigma_{\gamma}(u)=\Sigma(u)$
by~\eqref{eq:17} above (cf.\ Figure~\ref{fig:S} for illustration).
\end{proof}
We have the following comparison with the Average Value-at-Risk. The
comparison is sharp in the sense that the risk rates cannot be improved.
\begin{cor}
\label{cor:4}For every random variable~$X\in L^{1}$ it holds that
\begin{align}
e_{\frac{1}{2-\alpha}}(X) & \le\AVaR_{\alpha}(X),\qquad\alpha\in[0,1],\label{eq:3}\\
\shortintertext{\text{{and}}}\frac{\alpha}{3\alpha-1}\E X+\frac{2\alpha-1}{3\alpha-1}\AVaR_{2-\frac{1}{\alpha}}(X) & \le e_{\alpha}(X)\le\AVaR_{2-\frac{1}{\alpha}}(X)\label{eq:4}
\end{align}
for every $\alpha\in\left[\nicefrac{1}{2},1\right]$.

For non-negative random variables ($X\ge0$ a.s.)\ we further have
\begin{align}
\AVaR_{\alpha}(X) & \le\frac{1}{1-\alpha}e_{\frac{1}{2-\alpha}}(X)\label{eq:3-1}\\
\shortintertext{and}e_{\alpha}(X) & \le\frac{\alpha}{1-\alpha}\E X.\label{eq:12}
\end{align}
The risk rates in the preceding equations \eqref{eq:3}\textendash \eqref{eq:12}
are optimal, they cannot be improved.
\end{cor}

\begin{rem}
The preceding corollary might give the impression that $e_{\alpha}$
is \textquoteleft weak\textquoteright{} in the sense that it attains
smaller values than the average value at risk and is comparable to
the risk neutral expectation. However, it holds that $e_{\alpha}(X)\to1$
for $\alpha\to1$, as follows readily from~\eqref{eq:ExpectileFO}.
Further, we have that the Average Value-at-Risk is a lower bound for
the expectiles in view of~\eqref{eq:3-1}, so that expectiles are
at least as \textquoteleft strong\textquoteright{} as the Average
Value-at-Risk.
\end{rem}

\begin{proof}[Proof of Corollary~\ref{cor:4}]
Employing the notation of the proof of Proposition~\ref{prop:Risk}
and $\Sigma_{\alpha}(u)\coloneqq\min\left(1,\frac{1-u}{\frac{1}{\alpha}-1}\right)$,
we have that $\Sigma(u)\le\Sigma_{\alpha}(u)$. As in the proof above
we conclude that $\mathcal{R_{\alpha}}(X)\le\AVaR_{\alpha}(X)$ and
with~\eqref{eq:1-1} that~\eqref{eq:3}. The inequality~\eqref{eq:3}
is tight, as $\Sigma_{\gamma}^{\prime}(1)\xrightarrow[\gamma\to1]{}\Sigma^{\prime}(1)$.

As for the remaining inequality choose $\gamma=\frac{\alpha}{3\alpha-1}$
in~\eqref{eq:Kusuoka}, and replace~$\alpha$ by $\frac{1}{2-\alpha}$
in~\eqref{eq:4} to obtain~\eqref{eq:3}. 

The inequality $\min\left(1,\frac{1-u}{1-\alpha}\right)\le\frac{1}{1-\alpha}\Sigma_{\gamma}(u)$
is evident for every $u\in[0,1]$, and the remaining assertion~\eqref{eq:3-1}
follows by the same reasoning as above. However, for inequality~\eqref{eq:3-2}
to hold true it is essential that $X\ge0$ a.s. 

H\"older's inequality, applied to~\eqref{eq:8}, gives 
\[
\E X\le s_{\alpha}(X)\le\int_{0}^{1}F_{X}^{-1}(u)\,\mathrm{d}u\cdot\max_{u\in[0,1]}\frac{\alpha(1-\alpha)}{\bigl(\alpha-u(2\alpha-1)\bigr)^{2}}=\E X\cdot\frac{\alpha}{1-\alpha}
\]
and thus~\eqref{eq:12}.
\end{proof}

\section{\label{sec:Conditional}Conditional and Dynamic Risk Measure}

Risk functionals \textendash ~as discussed above~\textendash{} are
employed to assess the risk of a random outcome. For this reason,
they have the economic interpretation of an insurance premium, while
the random outcome is the random insurance benefit (the random variable).
While the premium is known beforehand, the insurance benefit (the
random outcome) is not, it is revealed later.

Conditional risk measures are employed in risk management over time,
they address stochastic processes instead of random variables. Nested
risk measures, which are compositions of risk functionals over time,
enjoy the economic interpretation of risk premiums for insurance on
a rolling horizon basis. For a discussion of nested risk functionals
we may refer to \citet{Cheridito2011,Riedel2004,Shapiro2012,Ruszczynski}
and \citet{PichlerSchlotterMartingales}.

\subsection{The conditional expectile\label{subsec:CondExpectile}}

Definition~\ref{def:Expectile} allows extending the expectile to
conditional expectiles, which are conditioned on some $\sigma$\nobreakdash-algebra.
This constitutes a major building block to extend the definition of
expectiles from random variables to stochastic processes.
\begin{defn}[Conditional expectiles]
Let $X\in L^{1}$ be a random variable and $\mathcal{G}$ be a sub
$\sigma$\nobreakdash-algebra of $\mathcal{F}$, $\mathcal{G}\subset\mathcal{F}$
and $\alpha$ a $\mathcal{G}$\nobreakdash-measureable variable with
values in $[0,1]$. The $\mathcal{G}$\nobreakdash-measureable random
variable~$Z$ satisfying 
\begin{equation}
\alpha\cdot\E\bigl(\bigl(X-Z\bigr)_{+}\mid\mathcal{G}\bigr)=(1-\alpha)\cdot\E\bigl(\bigl(Z-X\bigr)_{+}\mid\mathcal{G}\bigr)\qquad\text{a.s.}\label{eq:7}
\end{equation}
is called the \emph{conditional expectile} (i.e., the conditional
version of~\eqref{eq:ExpectileFO}) and denoted $Z=e_{\alpha}(X\mid\mathcal{G})$.
As usual for the conditional expectation, we shall also write $e^{\mathcal{G}}(X)\coloneqq e(X\mid\mathcal{G})$
and $e^{Y=y}(X)\coloneqq e(X\mid Y=y)$ for the conditional expectile
and its versions.
\end{defn}

The solution of the problem~\eqref{eq:7} exists and is unique for
the same reasons as for the usual expectile, and $e_{\alpha}(X\mid\mathcal{G})\in L^{1}$,
as $\bigl(e_{\alpha}(X\mid\mathcal{G})-X\bigr)_{+}$ and $\E\bigl(\bigl(X-e_{\alpha}(X\mid\mathcal{G})\bigr)_{+}\mid\mathcal{G}\bigr)$
exist in~\eqref{eq:7}.
\begin{rem}
Based on the properties of the conditional expectation (cf.\ Section~\ref{sec:Elicitable}),
we have the following properties of the conditional expectile.
\begin{enumerate}[noitemsep]
\item \label{enu:1c} $e_{\alpha}^{\mathcal{G}}(X)\le e_{\alpha}^{\mathcal{G}}(Y)$
a.e.\ for all $X\le Y$ almost everywhere, 
\item \label{enu:2c} $e_{\alpha}^{\mathcal{G}}(X+Y)\le e_{\alpha}^{\mathcal{G}}(X)+e_{\alpha}^{\mathcal{G}}(Y)$
a.e.,
\item \label{enu:3c} $e_{\alpha}^{\mathcal{G}}(\lambda\,X)=\lambda\,e_{\alpha}^{\mathcal{G}}(X)$
for all $\lambda>0$ and $\lambda$ which is $\mathcal{G}$\nobreakdash-measurable, 
\item \label{enu:4c} $e_{\alpha}^{\mathcal{G}}(c+X)=c+e_{\alpha}^{\mathcal{G}}(X)$
for all $\mathbb{R}$\nobreakdash-valued $c$ measurable with respect
to $\mathcal{G}$. 
\end{enumerate}
\end{rem}

In what follows, we shall consider the conditional expectile for a
single $\sigma$\nobreakdash-algebra first and discuss regression.
Next, we consider filtrations $\mathcal{F}=(\mathcal{F}_{t})_{t\in\mathcal{T}}$,
typically generated by a stochastic process $X=(X_{t})_{t\in\mathcal{T}}$.

\subsection{\label{sec:Expectile}Conditional expectiles in stochastic optimization
and regression}

Stochastic optimization and most typical problems in machine learning
(as the training of neural networks) as well as specific problems
in inverse problems (cf.\ \citet{LuPereverzev}) consider the problem
\begin{align}
\text{minimize } & f_{0}(x)\coloneqq\E f(x,\xi)\label{eq:Problem0}\\
\text{subject to } & x\in\mathcal{X},\nonumber 
\end{align}
where the objective is a risk neutral expectation, $f\colon\mathcal{X}\times\mathbb{R}^{m}\to\mathbb{R}$
is a function, $\mathcal{X}\subset\mathbb{R}^{d}$ is closed and $\xi$
is a random variable with values in~$\mathbb{R}^{m}$. Sample average
approximation builds on independent realizations~$\xi_{i}$ of identically
distributed random variable~$\xi$, $i=1,\dots$, to solve~\eqref{eq:Problem0}
in real world applications. To this end, the empirical version 
\[
\hat{f}_{n}(x)\coloneqq\frac{1}{n}\sum_{i=1}^{n}f(x,\xi_{i})
\]
 is considered instead of the expectation $\E f(x,\xi)$ in~\eqref{eq:Problem0}
for varying~$x\in\mathcal{X}$.

We consider the measure points (observations) $X\in\mathcal{X}$ to
be random (with measure $P$) as well and intend to \textquoteleft learn\textquoteright{}
the function~$f_{0}$ based on observations 
\begin{equation}
\bigl(X_{i},f(X_{i},\xi_{i})\bigr),\quad i=1,\dots,n,\label{eq:0}
\end{equation}
where $(X_{i},\xi_{i})$ are revealed jointly (cf.\ \citet{DentchevaKernels}
for further motivation in stochastic optimization and an alternative
approach); even more generally, we consider the iid observations 
\begin{equation}
(X_{i},f_{i}),\quad i=1,\dots n,\label{eq:16}
\end{equation}
which is~\eqref{eq:0} with~$f_{i}\coloneqq f(X_{i},\xi_{i})$. 

To model~\eqref{eq:16}, let~$\rho$ be the probability measure
of the joint distribution $(X,f)$ and denote the marginal measure
by $P(A)\coloneqq\rho(A\times\mathbb{R})$. Then there exists a regular
conditional probability kernel (cf.\ \citet{Kallenberg2002Foundations})
so that
\begin{equation}
\rho(A\times B)=\int_{A}\rho(f\in B|\,x)\,P(dx).\label{eq:rho}
\end{equation}

The bivariate measure~$\rho$ in~\eqref{eq:rho} is not an artifact.
Indeed, denote the conditional measures of~$f$ given~$X$ by the
Markov kernel $\rho\colon\mathcal{X}\times\mathcal{B}(\mathbb{R})\to[0,1]$,
that is, $\rho(f\in A\mid X=x)=\rho(x,A)$, then $(X,f)$ jointly
follow the composed measure~\eqref{eq:rho}, 
\[
(X,f)\sim\rho,
\]
and hence both approaches are equivalent.

For a random vector $(X,f)\in\mathbb{R}^{d}\times\mathbb{R}$ with
law~$\rho$ set 
\begin{equation}
f_{0}(x)\coloneqq\E(f\mid X=x);\label{eq:f0}
\end{equation}
 this definition notably corresponds to 
\[
f_{0}(x)=\E(f(X,\xi)\mid X=x)
\]
in the setting~\eqref{eq:0} above. For this reason, the stochastic
optimization problem~\eqref{eq:Problem0} is equivalent to\footnote{The essential infimum$\essinf(f\mid X)$ is the largest random variable~$g$,
measurable with respect to $\sigma(X)$ (the $\sigma$\nobreakdash-algebra
generated by $X$), so that $g\le f$, cf.~\citet[Definition~A.34]{Follmer2004}.
Measurability is the crucial difference in comparison to the (unconditional)
essential supremum in Footnote~\ref{fn:esssup}.} 
\begin{equation}
\essinf_{x\in\mathcal{X}}\E(f\mid X=x),\label{eq:Problem2}
\end{equation}
 where $(X,f)$ is a random variable with law $\rho$, provided that
$\operatorname{supp}P=\mathcal{X}$, where 
\[
\supp P\coloneqq\bigcap\big\{ A\colon A\text{ is closed and }P(A)=1\big\}
\]
 is the support.\footnote{\label{fn:Support-1}Cf.\ \citet{Rueschendorf} for the support of
the marginal measure~$P$.} 

Note, however, that not every random vector $(X,f)$ can be recast
as in~\eqref{eq:0} for a function~$f$ and a random~$\xi$. For
this reason, the problem formulation~\eqref{eq:Problem2} is more
general than the genuine problem~\eqref{eq:Problem0}.

\subsection{Risk assessment with conditional expectiles}

To incorporate risk in the assessment, consider the conditional expectation~\eqref{eq:f0}
and define 
\[
f_{\alpha}(x)\coloneqq e_{\alpha}(f\mid X=x),
\]
where $e_{\alpha}^{\sigma(X)}$ is the conditional expectile introduced
in Section~\ref{subsec:CondExpectile} above. Based on~\eqref{eq:5},
we have that 
\[
f_{\alpha}(x)\ge f_{0}(x),\qquad\text{for }\alpha\ge\nicefrac{1}{2},\ x\in\mathcal{X}.
\]

The function $f_{\alpha}$ intentionally \emph{overestimates} (overrates)
the risk-free assessment~$f_{0}$ and the surplus $f_{\alpha}-f_{0}$
is the amount attributed to risk aversion.

To solve the risk averse version of the stochastic optimization problem~\eqref{eq:Problem2},
\begin{align}
\text{minimize } & e_{\alpha}(f\mid X=x)\label{eq:Problem0-2-1}\\
\text{subject to } & x\in\mathcal{X},\nonumber 
\end{align}
just find an estimator for $\hat{e}_{\alpha}$ for $e_{\alpha}$ first
and then solve 
\begin{align}
\text{minimize } & \hat{e}_{\alpha}(x)\label{eq:Problem0-2-1-1}\\
\text{subject to } & x\in\mathcal{X}.\nonumber 
\end{align}
The substitute $\hat{e}_{\alpha}(\cdot)$ is chosen in an adequate
space of functions. \citet{DentchevaKernels} consider the Nadaraya\textendash Watson
kernel estimator to solve the problem. Here, we exploit the problem
by using reproducing kernel Hilbert spaces (RKHS) with kernel function~$k$,
where we may refer to \citet{Berlinet2004} for details. 
\begin{defn}
For a kernel function $k\colon\mathcal{X}\times\mathcal{X}\to\mathbb{R}$,
the RKHS space $\mathcal{H}_{k}$ is the completion of the functions
$f(x)=\sum_{i=1}^{\ell}w_{i}\,k(x,x_{i})$ with respect to the inner
product 
\[
\left\langle k(\cdot,x_{i})\mid k(\cdot,x_{j})\right\rangle =k(x_{i},x_{j}),\qquad i,j=1,\dots,\ell,
\]
where $x_{i}$ and $x_{j}\in\mathcal{X}$. 

The regularized problem is 
\begin{equation}
\text{minimize }\frac{1}{n}\sum_{i=1}^{n}\ell_{\alpha}\bigl(\hat{e}_{\alpha}(X_{i})-f_{i}\bigr)+\lambda\|\hat{e}_{\alpha}\|_{k}^{2},\label{eq:39}
\end{equation}
where $\hat{e}_{\alpha}(\cdot)\in\mathcal{H}_{k}$. It follows from
the generalized representer theorem (cf.\ \citet{Schoelkopf2001}),
that the function $e_{\alpha}$ is given by $e_{\alpha}(\cdot)=\frac{1}{n}\sum_{i=1}^{n}w_{i}\,k(\cdot,X_{i})$,
that is, the supporting points are exactly the points $X_{i}$, $i=1,\dots,n$,
where measurements $f_{i}$, $i=1,\dots,n$, are available. It might
be convenient in some situations to find the best approximation located
at the points $\tilde{x}_{j}$, $j=1,\dots,\tilde{n}$, that is, the
function 
\[
\hat{e}_{\alpha}(\cdot)=\frac{1}{\tilde{n}}\sum_{j=1}^{\tilde{n}}w_{j}k(\cdot,\tilde{x}_{j}),
\]
for fewer or special design points $\tilde{x}_{j}$, $j=1,\dots,\tilde{n}$.
We describe the equations for this generalized problem.
\begin{algorithm}[t]
	\KwIn{Measurements $(f_i,X_i)$, $i=1,\dots n$, and support points $\tilde x_j$, $j=1,\dots,\tilde n$.}

\KwOut{The weights $w_j$, $j=1,\dots,\tilde n$, of the function \begin{equation}\label{eq:Best}\hat e_\alpha(\cdot)=\frac1{\tilde n}\sum_{j=1}^{\tilde n} w_j k(\cdot,\tilde x_j)\end{equation}  minimizing~\eqref{eq:39}.}
Set
\begin{equation}
	K_{ij}\coloneqq k(X_i,\tilde x_j)
\end{equation} for $i=1,\dots n$ and $j=1,\dots \tilde n$, and \begin{equation}\tilde K_{ij}\coloneqq k(\tilde x_i, \tilde x_j)\end{equation} for $i,j=1,\dots,\tilde n$.

\While{change of the weights $w$ encountered}{
\For{$i=1$ \KwTo $n$}{update $A_{ii}\leftarrow \begin{cases}\alpha & \text{if }f_i\le\frac1{\tilde n}\sum_{j=1}^{\tilde n}w_j\, k(X_i,\tilde x_j),\\ 1-\alpha & \text{else}\end{cases}$}
update \begin{equation}\label{eq:14}
	w\leftarrow w-\Bigl( \frac\lambda{\tilde n^2} \tilde K+ \frac1{n^2\tilde n}K^\top A K\Bigr)^{-1}\cdot \Bigl(\frac\lambda{\tilde n^2} \tilde K w + \frac1{n^2\tilde n}K^\top A K w - \frac1{n\tilde n}K^\top A f\Bigr)\end{equation}}

\KwResult{The best approximating function~\eqref{eq:Best}.}

\caption{Newton-like iteration to solve~\eqref{eq:39}\label{alg:Newton}}
\end{algorithm}

The first order conditions of problem~\eqref{eq:39} for the weights
$w_{j}$, $j=1,\dots,\tilde{n}$, are 
\begin{align}
0=\frac{1}{n}\sum_{i=1}^{n}2 & \cdot\left(\frac{1}{\tilde{n}}\sum_{j^{\prime}=1}^{\tilde{n}}w_{j^{\prime}}k(X_{i},\tilde{x}_{j^{\prime}})-f_{i}\right)\cdot\left\{ \begin{array}{cc}
\alpha & \text{if }f_{i}\le\hat{e}_{\alpha}(X_{i})\\
1-\alpha & \text{if }f_{i}\ge\hat{e}_{\alpha}(X_{i})
\end{array}\right\} \cdot\frac{1}{\tilde{n}}k(\tilde{x}_{j},X_{i})\ +\nonumber \\
 & +2\frac{\lambda}{\tilde{n}^{2}}\sum_{j=1}^{\tilde{n}}w_{j}k(\tilde{x}_{i},\tilde{x}_{j}).\label{eq:44}
\end{align}
Define $\tilde{K}\coloneqq\bigl(k(\tilde{x}_{\ell},\tilde{x}_{j})\bigr)_{\ell,j=1}^{n}$,
$K\coloneqq\bigl(k(X_{i},\tilde{x}_{j})\bigr)_{i=1,j=1}^{n,\tilde{n}}$
and 
\begin{align*}
A(w) & \coloneqq\operatorname{diag}\bigl(a_{i}(w),\ i=1,\dots,n\bigr)\\
\shortintertext{\text{with entries}}a_{i}(w) & =\begin{cases}
\alpha & \text{if }f_{i}\le\frac{1}{\tilde{n}}\sum_{j=1}^{n}w_{j}k(X_{i},\tilde{x}_{j}),\\
1-\alpha & \text{if }f_{i}\ge\frac{1}{\tilde{n}}\sum_{j=1}^{n}w_{j}k(X_{i},\tilde{x}_{j})
\end{cases}
\end{align*}
on the diagonal. Then the equations~\eqref{eq:44} rewrite as 
\begin{equation}
\left(\frac{\lambda}{\tilde{n}^{2}}\tilde{K}+\frac{1}{n^{2}\tilde{n}}K^{\top}A(w)K\right)w=\frac{1}{n\tilde{n}}K^{\top}A(w)f.\label{eq:FixPoint}
\end{equation}
This equation is not linear in~$w$, as $A(w)$ depends in a nonlinear
way on $w$. However, the problem can be solved by inverting the matrix
to obtain a fixed point equation. With that, the equation can be iterated,
and the algorithm converges after finitely many iterations, cf.~\eqref{eq:14}
in Algorithm~\ref{alg:Newton}. Figure~\ref{fig:Expectile} displays
a typical result of expectile regression. \citet{SteinwartFarooq2018}
is a starting point in investigating convergence properties of the
expectile regression problem. 
\end{defn}

\begin{figure}
\centering{}\includegraphics[viewport=20bp 0bp 640bp 440bp,clip,width=0.7\textwidth]{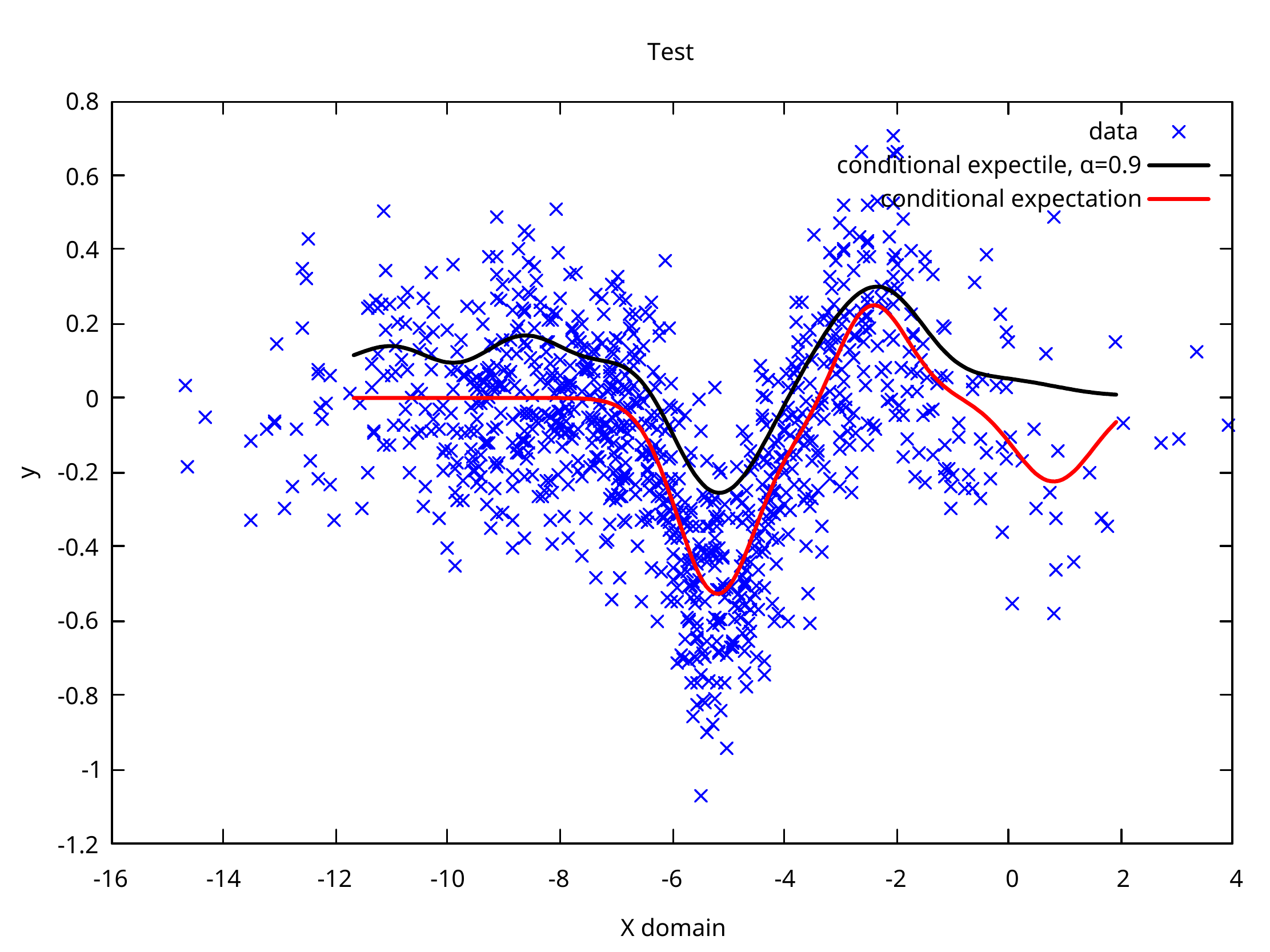}\caption{The expectile $\hat{e}_{90\,\%}(\cdot)$ based on $n=1000$ observations
overestimates the conditional expectation\label{fig:Expectile}}
\end{figure}

\begin{rem}
Note that the inverted matrix in~\eqref{eq:14} is the derivative
of the right-hand side with respect to~$w$, as $A$ is constant
for small changes in~$w$. For this reason, the iteration in Algorithm~\ref{alg:Newton}
is a Newton iteration in essence, although the function~\eqref{eq:ell}
is not differentiable. As $A(w)$ is constant for small variations
of~$w$, thus~\eqref{eq:14} vanishes locally.
\end{rem}

\section{\label{sec:Processes}Risk aversion in stochastic processes}

The considerations on the expectile in the preceding sections are based on random variables.
The conditional variant in the expectile regression is achieved with a single $\sigma$\nobreakdash-algebra.
In what follows, we generalize the expectile for stochastic processes
\textendash{} in a discrete time setting first, and then in continuous
time. 

\subsection{Nested expectile in discrete time}

Consider a stochastic process $X=(X_{t_{i}})_{i=0}^{n}$ in discrete
time, where $0\eqqcolon t_{0}<t_{1}<\dots<t_{n}=T$. For a dissection
in time consider the increments 
\[
X_{T}=X_{t_{0}}+(X_{t_{1}}-X_{t_{0}})+\dots+(X_{t_{n}}-X_{t_{n-1}}).
\]
The stochastic process $X$ is adapted to the filtration $\mathcal{F}$,
that is, $X_{t}$ is measurable for every $\mathcal{F}_{t}$, $t\ge0$,
so most often we just may choose $\mathcal{F}_{t_{i}}\coloneqq\sigma(X_{t_{j}}\colon j\le i)$.
As well, we shall denote the sequence of $\sigma$\nobreakdash-algebras
by $\mathcal{F}_{t_{0}\colon t_{n}}$.

In what follows, we shall associate a certain risk for the time period
$\Delta t\coloneqq t_{i+1}-t_{i}$ to come. For convenience in the
presentation in what follows, we introduce the \emph{rescaled} version
of the expectile as 
\[
\tilde{e}_{\beta}(\cdot)\coloneqq e_{\frac{1+\sqrt{\beta}}{2}}(\cdot)
\]
(i.e., $\alpha-\frac{1}{2}=\frac{\sqrt{\beta}}{2}$). The main reason
for the rescaling is that $e_{\nicefrac{1}{2}}(X)=\E X$, while $\AVaR_{0}(X)=\E X$,
e.g. To ensure consistent parametrizations with other risk measures,
we rescale the risk level so that $\tilde{e}_{0}(X)=\E X$ is associated
with the risk-free assessment, while $\tilde{e}_{1}(X)=\esssup X$
is the total risk averse assessment. The varying dynamic ($\sqrt{\beta}$
instead of $\beta$) turns out to be the natural choice in the continues
time situation addressed below.
\begin{defn}[Nested expectile]
\label{def:Nested}Let $\bigl(\Omega,\mathcal{F}=(\mathcal{F}_{t_{i}})_{i=1}^{n},P\bigr)$
be a filtered probability space and $\beta\colon\{t_{0},\dots,t_{n}\}\to[0,1]$
be stochastic process adapted to the filtration $\mathcal{F}=(\mathcal{F}_{t_{i}})_{i=1}^{n}$.
The nested expectile of the process with respect to the filtration
$\mathcal{F}_{t_{0}:t_{n}}$, denoted $\tilde{e}_{\beta(\cdot)}^{\mathcal{F}_{t_{0}:t_{n}}}$,
is 
\begin{equation}
\tilde{e}_{\beta(\cdot)}^{\mathcal{F}_{t_{0}:t_{n}}}(X)\coloneqq X_{0}+\tilde{e}_{\beta(t_{0},X_{t_{o}})\cdot(t_{1}-t_{0})}^{\mathcal{F}_{0}}\left(X_{t_{1}}-X_{t_{0}}+\dots+\tilde{e}_{\beta(t_{n-1},X_{t_{n-1}})\cdot(t_{n}-t_{n-1})}^{\mathcal{F}_{t_{n-1}}}\left(X_{T}-X_{t_{n-1}}\right)\right),\label{eq:Nested-1}
\end{equation}
or slightly more explicitly 

\[
\tilde{e}_{\beta(\cdot)}^{\mathcal{F}_{t_{0}:t_{n}}}(X)=X_{0}+\tilde{e}_{\beta(t_{0},X_{t_{0}})\cdot(t_{1}-t_{0})}^{\mathcal{F}_{0}}\begin{pmatrix}\begin{array}{l}
X_{t_{1}}-X_{t_{0}}+\\
\dots+\tilde{e}_{\beta(t_{n-2},X_{t_{n-2}})\cdot(t_{n-1}-t_{n-2})}^{\mathcal{F}_{t_{n-2}}}\begin{pmatrix}\begin{array}{l}
X_{t_{n-1}}-X_{t_{n-2}}\\
+\tilde{e}_{\beta(t_{n-1},X_{t_{n-1}})\cdot(t_{n}-t_{n-1})}^{\mathcal{F}_{t_{n-1}}}\left(X_{T}-X_{t_{n-1}}\right)
\end{array}\end{pmatrix}
\end{array}\end{pmatrix}.
\]
\end{defn}

Nested Risk measures have been considered by \citet{PhilpottMatosFinardi,PhilpottMatos},
e.g. In discrete time, fundamental properties of the Average Value-at-Risk
have been elaborated by \citet{ShapiroXin}, although for deterministic
risk rates only and for random variables instead of stochastic processes.
The definition above is dynamic, as the risk rate~$\beta$ is an adapted process itself.
Note that the risk rate at time~$t$ may be chosen to reflect the history of observations up to~$t$, it may
depend on $\{t_i\le t\colon i=1,\dots,n\}$.

We consider the following example, which prepares for the Wiener process.
\begin{example}[Random walk, cf.\ \citet{PichlerSchlotter2022}]
\label{exa:Walk}Consider a random walk process starting at $X_{0}$
with independent Markovian increments 
\begin{equation}
X_{t_{i+1}}-X_{t_{i}}\sim\mathcal{N}(0,t_{i+1}-t_{i})\label{eq:RandomWalk}
\end{equation}
and constant risk rate $\beta(t,x)=\beta$. With~\eqref{def:Nested}
and the asymptotic formula~\eqref{def:Nested} for the normal distribution,
we have that 
\begin{align}
X_{t_{1}}+\tilde{e}_{\beta}^{\mathcal{F}_{t_{1}}}(X_{t_{2}}-X_{t_{1}}) & =X_{t_{1}}+\sqrt{t_{i+1}-t_{i}}\sqrt{\frac{2}{\pi}}\sqrt{\beta(t_{i+1}-t_{i})}+o\bigl(t_{t+1}-t_{i}\bigr)\nonumber \\
 & =X_{t_{1}}+\sqrt{\frac{2\beta}{\pi}}(t_{i+1}-t_{i})+o\bigl(t_{t+1}-t_{i}\bigr).\label{eq:DriftWalk}
\end{align}
Nesting these expressions as in Definition~\ref{def:Nested} gives
the explicit expression 
\begin{equation}
\tilde{e}_{\beta(\cdot)}^{\mathcal{F}_{t_{0}:T}}(X)=X_{0}+\sqrt{\frac{2\beta}{\pi}}T+o(T),\label{eq:Drift}
\end{equation}
where $T$ is the terminal time, while 
\[
\tilde{e}_{0}^{\mathcal{F}_{t_{0}:T}}(X)=X_{0}
\]
for the risk rate $\beta=0$. The amount attributed to the risk averse
assessment in~\eqref{eq:Drift} thus accumulates linearly with time.
\end{example}

\begin{rem}[Tower property]
We emphasize as well that Definition~\ref{def:Nested} explicitly
involves time, the risk $\beta(t_{i})\cdot(t_{i+1}-t_{i})$ is associated
to the time interval starting at $t_{i}$ and ending at $t_{i+1}$.
With a further point in between, $t_{i+\nicefrac{1}{2}}$, the components
of the risk functionals above are 
\[
\tilde{e}_{\beta(t_{i})\cdot(t_{i+\nicefrac{1}{2}}-t_{i})}^{\mathcal{F}_{0}}\left(X_{t_{i+\nicefrac{1}{2}}}-X_{t_{i}}+\tilde{e}_{\beta(t_{i+\nicefrac{1}{2}})\cdot(t_{i+1}-t_{i+\nicefrac{1}{2}})}^{\mathcal{F}_{t_{i+\nicefrac{1}{2}}}}\left(X_{t_{i+1}}-X_{t_{i+\nicefrac{1}{2}}}\right)\right)
\]
instead of 
\[
\tilde{e}_{\beta(t_{i})\cdot(t_{i+1}-t_{i})}^{\mathcal{F}_{0}}\left(X_{t_{i+1}}-X_{t_{i}}\right).
\]
With that, the risk rates accumulate over time: accumulated risk rates
are $\beta(t_{i})(t_{i+\nicefrac{1}{2}}-t_{i})+\beta(t_{i+\nicefrac{1}{2}})(t_{i+1}-t_{i-\nicefrac{1}{2}})$
in the first case. This amount indeed coincides with $\beta(t_{i})(t_{i+1}-t_{i})$
(this is the risk rate in the second case), provided that $\beta(t_{i})=\beta(t_{i+\nicefrac{1}{2}})$,
i.e., the risk assessment does not vary over time.

For the expectation, the corresponding property is the tower property,
that is, $\E\bigl(\E X\mid\mathcal{G})\bigr)=\E X$. 
\end{rem}

\subsection{The nested expectile in continuous time}

In order to assign risk to a stochastic process in continuous time,
we consider the nested formulation introduced above for decreasing
time-steps. 
\begin{defn}[Nested expectile]
Let $X=(X_{t})_{t\le T}$ be a stochastic process adapted to $\mathcal{F}=(\mathcal{F}_{t})_{t\le T}$
and $\beta=(\beta_{t})_{t\le T}$ be c\`adl\`ag (i.e., right continuous,
with left limits) and adapted. With the nested expectile defined in
Definition~\ref{def:Nested}, the nested expectile is 
\begin{equation}
\tilde{e}_{\beta}^{\mathcal{F}}(X)=\lim_{\max\Delta t\to0}\tilde{e}_{\beta_{t_{0}:t_{n}}}^{\mathcal{F}_{t_{0}:t_{n}}}(X),\label{eq:Nested}
\end{equation}
provided that the limit with respect to decreasing mesh sizes $\max\Delta t\coloneqq\max_{i=1}^{n}t_{i+1}-t_{i}$
exists.
\end{defn}

\begin{example}[State independent risk rates]
Example~\ref{exa:Walk} generalizes for a state independent, but
time dependent Riemann integrable risk rate $\beta(x,t)=\beta(t)$.
As above, we obtain that
\[
X_{t_{1}}+\tilde{e}_{\beta}^{\mathcal{F}_{t_{1}}}(X_{t_{2}}-X_{t_{1}})=X_{t_{1}}+\sqrt{\frac{2\beta(t_{i})}{\pi}}(t_{i+1}-t_{i})+o\bigl(t_{t+1}-t_{i}\bigr)
\]
and thus 
\[
\tilde{e}_{\beta(\cdot)}^{\mathcal{F}_{t_{0}:T}}(X)=X_{0}+\sqrt{\frac{2}{\pi}}\int_{0}^{T}\sqrt{\beta(t)}\,\mathrm{d}t
\]
for $\Delta t\to0$, as $\beta$ is Riemann integrable. Again, this
is an explicit expression for the total risk aversion of the entire
random walk process with increments~\eqref{eq:RandomWalk}.
\end{example}

\begin{defn}[Risk generator]
Let $(X_{t})_{t\ge0}$ be a stochastic process adapted to the filtration
$\sigma(X)$ and $\beta(t,x)$ be a risk rate. The risk generator
is 
\[
\mathcal{G}_{\beta}f(x,t)\coloneqq\lim_{h\to0}\frac{\tilde{e}_{\beta(t,X_{t})}^{\sigma(X)}\bigl(f(X_{t+h})|\,X_{t}=x\bigr)-f(x)}{h},
\]
provided that the limit exists.
\end{defn}

Note, that $\mathcal{G}$ is an operator, which maps the (smooth)
function $f$ to $\mathcal{G}_{\beta}f$, which is a function again.
In contrast to the risk-neutral generator, the risk generator $\mathcal{G}_{\beta}$
is possibly not linear, as we will see in what follows.
\begin{prop}
\label{prop:Risk2}Let $X_{t}$ follow the stochastic differential
equation 
\begin{equation}
\mathrm{d}X_{t}=\mu(t,X_{t})\,\mathrm{d}t+\sigma(t,X_{t})\,\mathrm{d}W_{t}\label{eq:SDE}
\end{equation}
with respect to the Wiener process (Brownian motion) $(W_{t})_{t\ge0}$
and the functions $\mu$ and $\sigma$ be Lipschitz, i.e., $|\mu(t,x)-\mu(t,y)|+|\sigma(t,x)-\sigma(t,y)|\le K|x-y|$
so that strong solutions of~\eqref{eq:SDE} exist. For a smooth function~$f$,
the risk generator is 
\begin{align}
\mathcal{G}_{\beta}f(t,x) & =\frac{\partial f(t,x)}{\partial t}+\mu(t,x)\cdot\frac{\partial f(t,x)}{\partial x}+\frac{1}{2}\sigma(t,x)^{2}\cdot\frac{\partial^{2}f(t,x)}{\partial x^{2}}\nonumber \\
 & \qquad+\sqrt{\frac{2}{\pi}\beta(t,x)}\cdot\left|\sigma(x,t)\cdot\frac{\partial f(t,x)}{\partial x}\right|.\label{eq:Risk}
\end{align}
\end{prop}

\begin{proof}
The proof follows \citet[Section~7.3]{Oeksendal2003} (another valuable
reference is \citet{Karatzas1988}). 

Consider the stochastic process $Y_{t}\coloneqq f(t,X_{t})$. From
Ito's rule we deduce that 
\begin{align*}
Y_{t+\Delta t} & =Y_{t}+\int_{t}^{t+\Delta t}\left(\frac{\partial f}{\partial t}(s,X_{s})+\mu(s,X_{s})\frac{\partial f}{\partial x}(s,X_{s})+\frac{1}{2}\sigma(s,X_{s})^{2}\frac{\partial^{2}f}{\partial x^{2}}(s,X_{s})\right)\mathrm{d}s\\
 & \qquad+\int_{t}^{t+\Delta t}\sigma(s,X_{s})\frac{\partial f}{\partial x}(s,X_{s})\mathrm{d}W_{s},
\end{align*}
where the second part is a martingale with increments following the
Wiener process. Following the proof of the Ito formula in \citet[p.~46ff]{Oeksendal2003},
the functions $\mu$ and $\sigma$ are approximated by the constants
$\mu(s,X_{s})\approx\mu(t,X_{t})$ and $\sigma(s,X_{s})\approx\sigma(t,X_{t})$
for $s\in[t,t+\Delta t)$ so that 
\begin{align*}
Y_{t+\Delta t}-Y_{t}= & \left(\frac{\partial f}{\partial t}(t,X_{t})+\mu(t,X_{t})\frac{\partial f}{\partial x}(t,X_{t})+\frac{1}{2}\sigma(t,X_{t})^{2}\frac{\partial^{2}f}{\partial x^{2}}(t,X_{t})\right)\Delta t\\
 & \qquad+\sigma(t,X_{t})\frac{\partial f}{\partial x}(t,X_{t})\cdot\bigl(W_{t+\Delta t}-W_{t}\bigr).
\end{align*}
$Y_{t+\Delta t}-Y_{t}$ is a normally distributed random variable
with mean 
\[
Y_{t}+\left(\frac{\partial f}{\partial t}(t,X_{t})+\mu(t,X_{t})\frac{\partial f}{\partial x}(t,X_{t})+\frac{1}{2}\sigma(t,X_{t})^{2}\frac{\partial^{2}f}{\partial x^{2}}(t,X_{t})\right)\Delta t
\]
and variance
\[
\left(\sigma(t,X_{t})\frac{\partial f}{\partial x}(t,X_{t})\right)^{2}\Delta t.
\]
We deduce from~\eqref{eq:13} that 
\begin{align*}
\tilde{e}_{\beta\cdot\Delta t}^{X_{t}}(Y_{t+\Delta t})-Y_{t} & =\left(\frac{\partial f}{\partial t}(t,X_{t})+\mu(t,X_{t})\frac{\partial f}{\partial x}(t,X_{t})+\frac{1}{2}\sigma(t,X_{t})^{2}\frac{\partial^{2}f}{\partial x^{2}}(t,X_{t})\right)\Delta t\\
 & \qquad+\left|\sigma(t,X_{t})\frac{\partial f}{\partial x}(t,X_{t})\right|\sqrt{\Delta t}\cdot\sqrt{\frac{8}{\pi}}\left(\frac{1+\sqrt{\beta(t,x)\,\Delta t}}{2}-\frac{1}{2}\right).
\end{align*}
Now, by the definition of the risk generator~\eqref{eq:Nested},
we get the assertion.
\end{proof}
\begin{rem}
The drift \eqref{eq:DriftWalk} in Example~\ref{exa:Walk} now turns
out to be a specific case of the general relation revealed by~\eqref{eq:Risk},
both reveal the same pattern: any risk averse assessment adds the
additional drift term 
\[
\sqrt{\frac{2}{\pi}\beta(t,x)}\cdot\left|\sigma(x,t)\cdot\frac{\partial f(t,x)}{\partial x}\right|.
\]
For the absolute value $|\cdot|$ in the expression, the additional
drift term cannot be negative and always points in one direction,
the direction of risk. This is in line with risk aversion, as deviations
in the different directions are associated with profits and (for the
other direction) losses. Further, the coefficient $\beta$ models
the amount of local risk aversion.
\end{rem}

The behavior~\eqref{eq:Risk} has been found with other risk measures
as well, for example for the Entropic Value-at-Risk, cf.~\citet{PichlerSchlotter2022}.
For this reason, various results from the literature extend to the
nested expectile. 

\section{\label{sec:Control}The risk averse control problem}

While the classical theory on dynamic optimization builds on the risk-neutral
expectation (cf.\ \citet{Fleming1993}), we take risk into consideration
to the optimal control problem and derive a risk averse variant of
the Hamilton\textendash Jacobi\textendash Bellman equation. In what
follows we derive the governing equations formally by adapting the
presentation from \citet{PichlerSchlotter2022} for expectiles.

Consider the stochastic differential equation 
\begin{equation}
\mathrm{d}X_{t}^{u}=\mu\bigl(t,X_{t}^{u},u(t,X_{t}^{u})\bigr)\mathrm{d}t+\sigma\bigl(t,X_{t}^{u},u(t,X_{t}^{u})\bigr)\mathrm{d}W_{t}\label{eq:cSDE}
\end{equation}
driven by an adapted control policy $u(t,X_{t})$, where $u$ is a
measurable function. It is the objective to minimize the risk-averse
expectation of the accumulated costs, 
\[
\int_{t}^{T}c\bigl(s,X_{s},u(s,X_{s})\bigr)\mathrm{d}s+\Psi\bigl(X_{T}\bigr),
\]
where $\Psi(\cdot)$ is a terminal cost. Recall that the nested expectiles
accumulate costs and risk so that it is the objective to minimize
the value function 
\[
V^{u}(t,x)\coloneqq\tilde{e}_{\beta(\cdot)}^{\sigma(X)}\left(\left.\int_{t}^{T}\bigl(s,X_{s}^{u},u(s,X_{s}^{u})\bigr)\mathrm{d}s+\Psi\bigl(X_{T}^{u}\bigr)\right|X_{t}^{u}=x\right)
\]
among all policies $u\in\mathcal{U}$ chosen in a suitable set, where
$X_{t}^{u}$ solves the stochastic differential equation~\eqref{eq:cSDE}
for the policy~$u$. 
\begin{prop}
The value function
\[
V(t,x)\coloneqq\inf_{u(\cdot)\in\mathcal{U}}V^{u}(t,x),
\]
solves the differential equation
\begin{equation}
\frac{\partial V}{\partial t}(t,x)=\mathcal{H}_{\beta}\left(t,x,\frac{\partial V}{\partial x},\frac{\partial^{2}V}{\partial x^{2}}\right)\label{eq:HJB}
\end{equation}
with terminal condition $V(T,x)=\Psi(x)$, where 
\begin{equation}
\mathcal{H}_{\beta}(t,x,g,A)\coloneqq\sup_{u\in U}\left\{ -c(t,x,u)-g\cdot\mu(t,x,u)-\frac{1}{2}A\,\sigma(t,x,u)^{2}-\sqrt{\frac{2}{\pi}\beta(t,x)}\cdot\bigl|g\cdot\sigma(t,x,u)\bigr|\right\} \label{eq:Hamilton}
\end{equation}
is the Hamiltonian, cf.\ \citet[Section~IV,~(3.2)]{Fleming1993}. 
\end{prop}

To accept the assertion recall that 
\begin{align*}
\frac{1}{h}\tilde{e}_{\beta(\cdot)}^{\sigma(X)} & \left(\left.\int_{t}^{t+h}c\bigl(s,X_{s}^{u},u(s,X_{s}^{u})\bigr)\mathrm{d}s+V(t+h,X_{t+h})-V(t,x)\right|X_{t}^{u}=x\right)\\
 & \xrightarrow[h\to0]{}c(t,x,u)+\mathcal{G}_{\beta}V(t,x)
\end{align*}
by the definition of the risk generator. While the left-hand side
vanishes by the dynamic programming principle for the optimal policy,
it follows for the right-hand side that 
\[
0=\inf_{u\in U}c(t,x,u)+\mathcal{G}_{\beta}V(t,x).
\]
With Proposition~\ref{prop:Risk2}, this leads to the equation~\eqref{eq:HJB}
with Hamiltonian~\eqref{eq:Hamilton}.

\medskip{}

The fundamental equation~\eqref{eq:HJB} is the Hamilton\textendash Jacobi\textendash Bellman
(HJB) partial differential equation. It is essential to observe that
the HJB equation has the additional term 
\[
\sqrt{\frac{2}{\pi}\beta(t,x)}\cdot\left|\sigma(t,x,u)\frac{\partial V}{\partial x}\right|
\]
 involving the gradient; the total gradient in the Hamiltonian~\eqref{eq:HJB}
thus comes with the coefficient 
\[
\mu(t,x,u)+\sqrt{\frac{2}{\pi}\beta(t,x)}\cdot\sigma(t,x,u)\cdot\sign\left(\sigma(t,x,u)\frac{\partial V}{\partial x}\right).
\]
 That is, risk aversion increases the trend~$\mu$ by the amount
$+\sqrt{\frac{2}{\pi}\beta(t,x)}\cdot\sigma(t,x,u)$, while letting
the volatility~$\sigma$ of the process unaffected. 

In typical situations, $\frac{\partial V}{\partial x}$ does not change
its sign. For this reason, the classical theory on viscosity solutions
on existence of solutions of~\eqref{eq:HJB} applies directly, without
modifications. As well, explicit solutions of specific equations are
known. In these situations, the explicit results can be adapted to
the risk averse situation, cf.\ \citet{PichlerSchlotter2020} for
applications from financial mathematics.

\section{\label{sec:Summary}Summary}

This paper exploits the unique properties of expectiles in stochastic
and in dynamic optimization. We start by giving tight comparisons
with common risk measures first. Next, we define the conditional expectile.
The conditional expectile can be nested to extend the scope of risk
functionals (risk measures) to stochastic processes in discrete and
in continuous time. For the random walk process or stochastic processes
driven by a stochastic differential equation, explicit evaluations
of the nested risk functional are available.

The risk generator is defined in analogy to the generator for stochastic
processes. The risk generator involves an additional term which is
caused by risk. With that, the risk generator is a non-linear differential
operator. The aspect of risk augments the Hamiltonian via an additional
term, which is responsible for risk only and the risk averse Hamilton\textendash Jacobi\textendash Bellman
equations thus derive accordingly.\bibliographystyle{abbrvnat}
\bibliography{/home/alopi/Dropbox/Literatur/LiteraturAlois}

\section{Appendix\label{sec:Appendix}}

The function $x\mapsto(1-\alpha)\E(x-X)_{+}-\alpha\E(X-x)_{+}$ has
slope 
\begin{align*}
(1-\alpha)\,P(X\le x)+\alpha\,P(X\ge x) & =\alpha+(1-2\alpha)P(X\le x)\\
 & \ge\alpha P(X\le x)+(1-2\alpha)P(X\le x)\\
 & =(1-\alpha)P(X\le x)\\
 & \ge0
\end{align*}
and is therefore strictly increasing for every $\alpha\in(0,1)$ in
the support of~$X$ so that the expectile is unique. Further, the
slope is so that the function is convex for $\alpha\le\nicefrac{1}{2}$
and concave for $\alpha\ge\nicefrac{1}{2}$.

Denote by $x_{\alpha}$ ($y_{\alpha}$, resp.)\ the expectile for
$X$ ($Y$, resp.), i.e., 
\begin{align*}
\alpha\E(X-x_{\alpha})_{+} & =(1-\alpha)\E(x_{\alpha}-X)_{+}\text{ and}\\
\alpha\E(Y-y_{\alpha})_{+} & =(1-\alpha)\E(y_{\alpha}-Y)_{+}.
\end{align*}
With $x_{+}-(-x)_{+}=x$ we have further
\begin{align*}
\big(\alpha-\frac{1}{2}\big)\E(X-x_{\alpha})_{+} & -\big(\frac{1}{2}-\alpha\Big)\E(x_{\alpha}-X)_{+}=\frac{1}{2}\E(X-x_{\alpha})\text{ and}\\
\Big(\alpha-\frac{1}{2}\Big)\E(Y-y_{\alpha})_{+} & -(\frac{1}{2}-\alpha)\E(y_{\alpha}-Y)_{+}=\frac{1}{2}\E(Y-y_{\alpha}).
\end{align*}
For $\alpha\ge\nicefrac{1}{2}$ we obtain by convexity of the function
$x\mapsto x_{+}$ that 
\begin{align*}
\big(\alpha-\frac{1}{2}\big)\E(X+Y-x_{\alpha}-y_{\alpha})_{+} & \le\big(\alpha-\frac{1}{2}\big)\E(X-x_{\alpha})+\big(\alpha-\frac{1}{2}\big)\E(Y-y_{\alpha})\\
 & =\frac{1}{2}\E(X-x_{\alpha})+\big(\frac{1}{2}-\alpha\Big)\E(x_{\alpha}-X)_{+}+\frac{1}{2}\E(Y-y_{\alpha})+(\frac{1}{2}-\alpha)\E(y_{\alpha}-Y)_{+}\\
 & \le\frac{1}{2}\E(X-x_{\alpha})+\frac{1}{2}\E(Y-y_{\alpha})+\big(\frac{1}{2}-\alpha\Big)\E(x_{\alpha}+y_{\alpha}-X-Y)_{+}.
\end{align*}
It follows that 
\[
\alpha\E(X+Y-x_{\alpha}-y_{\alpha})_{+}\le(1-\alpha)\E(x_{\alpha}+y_{\alpha}-X-Y)_{+}.
\]
The assertion follows by monotonicity again.
\end{document}